%% file: eff2.tex
\documentclass[11pt,reqno]{amsart}

\usepackage[margin=3cm]{geometry}
\usepackage{amsmath,amssymb,latexsym}
\usepackage{amsthm}
\usepackage{bm}
\usepackage{bbm,stmaryrd}
\usepackage{url}
\usepackage[colorlinks=true,linkcolor=blue,anchorcolor=blue,citecolor=blue,
     		filecolor=blue,urlcolor=blue]{hyperref}
\usepackage{doi}

\newcommand{\A}{\ensuremath{\mathbb{A}}}

\newcommand{\catC}{\ensuremath{\mathbb{C}}}
\newcommand{\N}{\ensuremath{\mathbb{N}}}
\newcommand{\G}{\ensuremath{\mathbb{G}}}

\newcommand{\B}{\ensuremath{\mathbb{B}}}
\newcommand{\X}{\ensuremath{\mathbb{X}}}
\newcommand{\PP}{\ensuremath{\mathcal{P}}}
\newcommand{\Asm}{\ensuremath{\mathsf{Asm}}}
\newcommand{\Gpd}{\ensuremath{\mathsf{Gpd}}}

\newcommand{\EE}{\ensuremath{\mathcal{E}}}
\newcommand{\Eff}{\ensuremath{\mathcal{E}\!f\!f}} 
\newcommand{\Mod}{\ensuremath{\mathsf{Mod}}} 
\newcommand{\psh}[1]{\ensuremath{\mathsf{Set}^{#1^{\mathsf{op}}}}}
\newcommand{\Set}{\ensuremath{\mathsf{Set}}}

\newcommand{\ul}[1]{\ensuremath{\underline{#1}}}
\newcommand{\pathg}[1]{\ensuremath{{#1}^{\downarrow}}}
\newcommand{\tDelta}{\ensuremath{\tilde{\Delta}}}

\newcommand{\hPP}{\ensuremath{\widehat{\PP}}}

\newcommand{\y}{\ensuremath{\mathsf{y}}} 
\newcommand{\yon}{\ensuremath{\mathsf{y}}} 

\newcommand{\exlex}[1]{\ensuremath{{#1}_{\mathsf{ex/lex}}}}
\newcommand{\exreg}[1]{\ensuremath{{#1}_{\mathsf{ex/reg}}}}
\newcommand{\reglex}[1]{\ensuremath{{#1}_{\mathsf{reg/lex}}}}

\newcommand{\op}[1]{\ensuremath{{#1}^{\mathrm{op}}}}
\newcommand{\dom}{\ensuremath{\mathsf{dom}}}
\newcommand{\cod}{\ensuremath{\mathsf{cod}}}

\newcommand{\mono}{\rightarrowtail}
\newcommand{\epi}{\twoheadrightarrow}
\newcommand{\toto}{\rightrightarrows}
\newcommand{\hook}{\ensuremath{\hookrightarrow}}

\newcommand{\cof}{\ensuremath{\rightarrowtail}}
\newcommand{\fib}{\ensuremath{\twoheadrightarrow}}
\renewcommand{\to}{\ensuremath{\rightarrow}}

\newcommand{\onto}{\ensuremath{\twoheadrightarrow}}

\newcommand{\HH}{\ensuremath{\mathbb{H}}}

\newcommand{\F}{\ensuremath{\mathbb{F}}}
\newcommand{\K}{\ensuremath{\mathbb{K}}}

\newcommand{\pair}[1]{\ensuremath{\langle #1\rangle}}

\newcommand{\Id}{\mathsf{Id}}
\newcommand{\id}[1]{\Id_{#1}}

\newtheorem{theorem}{Theorem}
\newtheorem*{theorem*}{Theorem}
\newtheorem{proposition}[theorem]{Proposition} 
\newtheorem{lemma}[theorem]{Lemma}
\newtheorem{corollary}[theorem]{Corollary} 

\theoremstyle{remark}
 
\newtheorem*{remarks*}{Remarks}

\theoremstyle{definition}
\newtheorem{definition}[theorem]{Definition}

\usepackage{tikz}
\usepackage{pdfpages}
\usepackage{tikz-cd}
\newcommand{\pbmark}{\ar[dr, phantom, "\lrcorner" very near start, shift right=.5ex]}	

\begin{document}

\title{Toward the effective 2-topos}
\author{Steve Awodey \and Jacopo Emmenegger}
\date{\today}
\dedicatory{to Pino Rosolini on the occasion of his 70th birthday}
\begin{abstract}
A candidate for the effective 2-topos is proposed and shown to include the effective 1-topos as its subcategory of 0-types.
\end{abstract}
\maketitle

\section{Introduction}

In the effective topos $\Eff$ ~\cite{Hyland1982} all maps $f : \N\to \N$ are computable.  Even more remarkably, $\Eff$ contains a small, full subcategory $\Mod$ that is internally complete but is \emph{not} a poset~\cite{HRR1990,Hyland1988}, a combination that is not possible in the classical logic of $\Set$, by a result of~\cite{Freyd:abcat}.  The elementary  topos $\Eff$ is not Grothendieck. Indeed, it is not even cocomplete;  for example, it lacks the countable coproduct  $\coprod_{n\in\N}1_n$.

A higher-dimensional version of $\Eff$ should provide an example of a non-Grothendieck elementary higher-topos, a concept  under current investigation~\cite{Shulman2018, Rasekh2022, Anel2025}.  One's first thought might be to simply take simplicial objects in $\Eff$, perhaps equipped with the Kan-Quillen model structure; but this fails for several reasons: 
\begin{enumerate}

\item\label{prob:constr} Some aspects of the usual Kan-Quillen model structure on simplicial sets  $\Set^{\op{\Delta}}$ rely on the classical logic of $\Set$, e.g.\  for exponentials $Y^X$ of Kan complexes to again be Kan (as was already determined by Stekelenburg~\cite{Stekelenburg2016}), see~\cite{BezemCoquand2015,BCP2015}.
 
\item\label{prob:pity} But the constructive version of the Kan-Quillen model structure on simplicial sets $\Set^{\op{\Delta}}$, given by \cite{GambinoHenry2022, GSS2022, Henry2020}, does not directly model the $\Pi$-types without some further modifications, since not all objects are cofibrant.

\item\label{prob:zeroty} Using one of the cubical models of type theory, such as \cite{CCHM:2018ctt}, \cite{ABCHFL}, one can produce constructive versions of the Kan-Quillen model structure that do model $\Pi$-types, as in~\cite{AwodeyCCMC, ACCRS}, but there is still a more fundamental problem: $\Eff$ is already an exact completion, and $\Eff^{\op{\Delta}}$ and $\Eff^{\op{\Box}}$ are, in some sense, higher exact completions of that; so we should not expect the 1-topos of 0-types in, say, the higher topos $\Eff^{\op{\Box}}$ (if it is one) to be equivalent to $\Eff$ itself. And indeed, it has been shown that the category of $0$-types in the model in cubical assemblies of \cite{AFS,Uemura:2018} is not equivalent to $\Eff$, as follows from the failure there of Church's thesis \cite{SwanUemura2021}.  
\end{enumerate}

Toward solving these problems, we first review how $\Eff$ is determined as the ($1$-)exact completion of the left exact category $\PP$ of partitioned assemblies. We then use that analysis to construct a (putative) $(2,1)$-exact completion $\Eff_2$ of $\PP$ that includes the 1-exact completion $\exlex{\PP} = \Eff$ as the subcategory of 0-types, which will then enjoy the same internal logic as the latter.  We remark that our meta-theory is classical, like that of the usual effective 1-topos, so that e.g.\ the category of 0-types in the standard model of type theory in Kan simplicial sets \cite{KLV:21} is (equivalent to) the category of sets, not that of setoids as in \cite{Shulman2023:derstd}.

\section{The effective 1-topos}

Recall from~\cite{RR1990} the construction of the effective topos $\Eff$ as an exact completion of the finitely complete category $\PP$ of partitioned assemblies. Recall also that this construction can be factored in two steps~\cite{Carboni1995}. First, the category of \emph{assemblies} $\Asm$ is the regular completion of $\PP$ (preserving finite limits), and then the category $\Eff$ is the exact completion (preserving finite limits and regular epimorphisms) of the regular category $\Asm$:
\begin{equation}\label{diag:subcats_of_psh}
\PP \hook \reglex{\PP} = \Asm \hook \exreg{\Asm}= \exreg{(\reglex{\PP})} = \exlex{\PP} = \Eff
\end{equation}

Let us consider each step, regarded as a subcategory of the (small) colimit completion, the category of presheaves $\widehat{\PP} = \psh{\PP}$~\cite{HuTholen:1996}:
\begin{itemize}
 \item  $\PP \hook \reglex{\PP} = \Asm \hook \widehat{\PP}$  adds all those presheaves $P/K$ that are \emph{regular images} (``kernel quotients'') of maps $P\to Q$ in $\PP$:
\[
\begin{tikzcd}
K \ar[r, shift right = .5ex] \ar[r, shift left = .5ex] & P \ar[rr] \ar[rd, two heads] && Q \\
	&& P/K \ar[ru, tail] &
\end{tikzcd}
\]
\item $\Asm \hook \exreg{\Asm} = \Eff \hook \widehat{\PP}$ consists of those presheaves $A/E$ that are quotients of equivalence relations $E\rightrightarrows A$ in $\Asm$:
\[
\begin{tikzcd}
E \ar[r, shift right = .5ex] \ar[r, shift left = .5ex] & A \ar[r, two heads] & A/E
\end{tikzcd}
\]
\end{itemize}

Now let us refine the factorization \eqref{diag:subcats_of_psh}, namely 
\[
\begin{tikzcd}
\PP \ar[r, hook] & \Asm \ar[r, hook] & \Eff \ar[r, hook] & \widehat{\PP}\,,
\end{tikzcd}
\]
by interpolating the following intermediate full subcategories:
\begin{equation}\label{diag:subcats_of_psh_simple}
\begin{tikzcd}
\PP \ar[rrrrr,hook] \ar[d,hook] &&&&& \widehat{\PP}\\
\mathsf{IndProj} \ar[r, hook] & \Asm \ar[r, hook] & \mathsf{Coh} \ar[r, hook] & \Eff \ar[r, hook] & \mathsf{Cpt} \ar[r, hook] & \mathsf{Sh}(\Asm) \ar[u, hook] 
\end{tikzcd}
\end{equation}

\begin{itemize}
 \item  $\mathsf{IndProj}$ consists of the \emph{indecomposable projectives}: projective objects $P \in \widehat{\PP}$ such that $0 \ncong P$, and $X+Y \cong P$ implies $X\cong P$ or $Y\cong P$.  These are just the representable objects $\yon{P}$ for $P\in\PP$, so the first step above is an equivalence of categories.  
 
 \item A presheaf $K$ is called \emph{(super-)compact}~\cite[Remark~D3.3.10]{Johnstone:2002seII} if every jointly epimorphic family $(E_i \to K)_i$ contains a single epimorphism $E_k \epi K$.   It follows that these are the objects $K$ that have a cover $\yon{P} \onto K$ by a representable.  Let $\mathsf{Cpt}$ be the full subcategory of compact presheaves (dropping the ``super-'' for the remainder of this paper).
 
\item $ \mathsf{Coh} $ consists of the \emph{(super-)coherent} presheaves:  those that are both compact and have a compact diagonal map $\Delta_C : C \to C \times C$, where in general a map $f : Y\to X$ is said to be \emph{(super-)compact} if for every compact object $K$ and map $K\to X$, the pullback $K' = K\times_X Y$ is a compact object:
\[
\begin{tikzcd}
K' \ar[r] \ar[d]  \pbmark & Y \ar[d, "f"]\\
K \ar[r] & X
\end{tikzcd}
\]
Equivalently, every pullback over a representable $\y{P}$  is compact.  Note that this is actually stronger than the usual notion of coherence due to our convention on the notion of (super-)compactness.

\item By $\mathsf{Sh}(\Asm)$ we mean the sheaves on the category $\Asm$ of assemblies for the \emph{regular epimorphism} topology, in which the covering sieves are those that contain a regular epimorphism. Since $\Asm$ is the regular completion of $\PP$, it is easy to see that the sheaves on these are just the presheaves on the latter, $\widehat{\PP} \cong \mathsf{Sh}(\Asm) \hook \widehat{\Asm}$.  It will nonetheless be convenient to have the description of $\Eff$ as a subcategory of $\mathsf{Sh}(\Asm)$.  
\end{itemize}

\begin{proposition}[\cite{Lack1999}]\label{prop:Lack}
For a regular category $\mathcal{R}$, the exact completion (preserving regular epis) can be described as the full subcategory of sheaves for the regular topology,
\[
\exreg{\mathcal{R}} \hook \mathsf{Sh}(\mathcal{R}, \mathsf{reg}) \,,
\]
on those objects $E$ with an \emph{exact presentation} by representables: thus those for which there is a kernel-quotient diagram
\[
\yon{R'} \rightrightarrows \yon{R} \epi E\,,
\]
in which $R, R'\in \mathcal{R}$ and $\yon{R'} \rightrightarrows \yon{R}$ is the kernel pair of $\yon{R} \epi E$, and $\yon{R} \epi E$ is the coequalizer of $\yon{R'} \rightrightarrows \yon{R}$.
\end{proposition}

\begin{proposition}\label{prop:CohIsEff}
 The category $\mathsf{Coh}$ is equivalent to $\Eff$.
\end{proposition}
\begin{proof}
 The coherent presheaves are those $C$ that are compact $\yon{P} \epi C$ and have a compact diagonal $\Delta_C : C\to C\times C$.  Such a presheaf therefore has an exact presentation $K \toto \yon{P} \epi C$ by the objects $K$ and $\yon{P}$, each of which is an assembly: $\yon{P}$ because it is a partitioned assembly, and $K$ because it is compact, and therefore has a representable cover $\yon{Q} \epi K$, and also has a mono $K \mono \yon{P} \times\yon{P} \cong \yon{(P \times P)}$.  Conversely, if an object $C$ has an exact presentation $A \toto B \epi C$ by assemblies $A$ and $B$, then it is compact (because $B$ is compact). Let us show that $\Delta_C : C\to C\times C$ is compact.  We have a pullback square:
 \[
\begin{tikzcd}
A \ar[d]  \ar[r] \pbmark & C \ar[d, "\Delta"] \\
B\times B \ar[r, two heads]  & C \times C
\end{tikzcd}
\]
and if $K$ is any compact object and $c : K \to  C \times C$, then we need to show that (the domain of) $c^*\Delta$ is a compact object.  Since $K$ is compact, there is cover $k : \yon{P} \epi K$ and since $\yon{P}$ is projective, $ck: \yon{P} \to C \times C$ lifts across the epi $B\times B \epi C \times C$ via, say, $\ell : \yon{P} \to B \times B$.
\[
\begin{tikzcd}
& A \ar[d]  \ar[r] \pbmark & C \ar[d, "\Delta"] \\
& B\times B \ar[r, two heads]  & C \times C\\
\yon{P} \ar[ru, dotted, "\ell"] \ar[r, two heads, "k"] & K \ar[ru, "c"]
\end{tikzcd}
\]
It follows that $c^*\Delta$ is covered by $\ell^*A$.  But since $\yon{P}, A, B$ are all assemblies, so is the pullback $\ell^*A$, which is therefore covered by a representable, and so is compact.  Therefore $c^*\Delta$ is also compact.
 
Thus the coherent presheaves are equivalently described as those objects in $\hPP$ with an exact presentation by assemblies. But these are the representable sheaves in  $\mathsf{Sh}(\mathcal{R}, \mathsf{reg})$ when $\mathcal{R} = \Asm$.   By the foregoing Proposition this is $\exreg{\Asm} \cong \Eff$.
\end{proof}

Summarizing this section, we have the situation:
\[
\begin{tikzcd}
\PP \ar[r, hook] & \Asm \ar[r, hook] & \mathsf{Coh}\mathsf{Sh}(\Asm) \ar[r, hook] \ar[d,swap, "\sim"] & \mathsf{Sh}(\Asm) \ar[d,"\sim"]\\
		&& \Eff \ar[r, hook] & \widehat{\PP} 
\end{tikzcd}
\]
 We can recover the familiar description of $\Eff$ as the exact completion of $\PP$ described in terms of \emph{pseudo-equivalence relations} in $\PP$ (as in~\cite{Carboni1995}), from the above specification of coherent objects. Indeed, if $C\in \widehat{\PP}$ is coherent, then because it is compact, there is a representable cover $\yon{P} \onto C$, the kernel $K \rightrightarrows \yon{P}$ of which has $K$ compact (since the diagonal of $C$ is compact), and therefore $K$ is in turn covered  by a representable $\yon{Q} \onto K$. Since $K \rightrightarrows \yon{P}$ is an (actual) equivalence relation and $\yon{Q}$ is projective, it is easy to see that the resulting graph $\yon{Q} \rightrightarrows \yon{P}$ is a \emph{pseudo}-equivalence relation in $\widehat{\PP}$, and since it is in the image of the Yoneda embedding, its preimage $Q \rightrightarrows P$ is a pseudo-equivalence relation in $\PP$.

\section{Coherent groupoids}

In order to obtain a 2-topos containing the category $ \mathsf{Coh}\mathsf{Sh}(\Asm)$ of coherent sheaves as the 0-types, we shall take the internal \emph{groupoids} in $\mathsf{Sh}(\Asm) \simeq \widehat{\PP}$, generalizing the equivalence relations of Proposition~\ref{prop:Lack}.  We then cut down to the \emph{coherent} groupoids, so that the $0$-types are the coherent sheaves, and therefore equivalent to $\Eff$.

In other words, over the category $\hPP$ of presheaves, we shall have a pullback diagram of full subcategories:
\begin{equation}\label{diag:subcats}
\begin{tikzcd}
\Eff \ar[d, hook]  \ar[r, hook] \pbmark & \mathsf{Coh}\mathsf{Gpd}   \ar[d, hook] \\
\mathsf{Gpd}_0 \ar[r, hook]  & \mathsf{Gpd}
\end{tikzcd}
\end{equation}

We shall make frequent use of the fact that internal groupoids in the category $\hPP$ are the same thing as presheaves of groupoids, 
\[
\Gpd(\,\hPP\,) \cong [\op{\PP},  \Gpd]\,.
\]
To fix terminology, by a  \emph{(weak) equivalence} of such groupoids $e : \F \to \G$ we shall mean an internal functor that is essentially surjective on objects and fully faithful---internal conditions which can also be checked objectwise in $P\in\PP$.  A  \emph{strong (or homotopy) equivalence} is a functor that has a quasi-inverse---which of course need not obtain just because each $e_P : \F(P) \to \G(P$) has one.
 
\begin{lemma}\label{lem:disc_gpd_eqv}
Given a groupoid $\G=(G_1\rightrightarrows G_0)$ in $\hPP$,
a functor $\G\to \ul{X}$ into a discrete groupoid $\ul{X} = (\Delta : X\to X\times X)$ is an equivalence if, and only if, the map $G_0\epi X$ is epic and $G_1\toto G_0 $ is its kernel pair.
In that case, $\G$ is an equivalence relation, and $$\pi_0\G\cong G_0/G_1 \cong X$$ is its quotient.
A functor $\G\to \ul{X}$ is moreover a strong equivalence if, and only if, in addition, the map $G_0\epi X$ has a section.
\end{lemma}

\begin{proof}
The first claim is just a rewording of the definition:
an essentially-surjective-on-objects functor $e: \G\to \ul{X}$ into a discrete groupoid $\ul{X}$ is surjective on objects,
and it is fully faithful precisely when the square below is a pullback.
\[\begin{tikzcd}[column sep=7ex]
	G_1 \ar[d] \ar[r,"e_1"]  &  X  \ar[d, tail, "\Delta"]
	\\
	G_0\times G_0  \ar[r,swap,"{e_0\times e_0}"]  &  X\times X
\end{tikzcd}\]

If $e$ is a strong equivalence, then it clearly has a section.
Conversely, a section $s$ becomes a weak inverse using the fact that the square above is a pullback.
\end{proof}

\begin{lemma}\label{lemma:pspb_gpd}
Consider the two commutative diagrams of groupoids below,
where $\tDelta = \langle \dom, \cod \rangle : \pathg{\G} \to\G\times\G$ is the canonical functor from the path groupoid (the arrow category)
and the lower square in the right-hand diagram is a pullback.
The square on the left is a pseudo-pullback if the indicated comparison functor $k$ is a strong equivalence.
\[\begin{tikzcd}
	\A  \ar[d] \ar[r]  &  \G  \ar[d,"\Delta"]
	\\
	\B \ar[r]  &  \G \times \G
\end{tikzcd}
\hspace{9ex}
\begin{tikzcd}
	\A  \ar[dd,bend right=10ex] \ar[d,dotted,"k"] \ar[r]
	&  \G  \ar[dd,"\Delta",bend left=10ex] \ar[d]
	&\\
	\A'  \ar[r] \ar[d] \pbmark  &  \pathg{\G}  \ar[d,swap,"\tDelta"]
	\\
	\B  \ar[r]  &  \G \times \G
\end{tikzcd}\]
\end{lemma}

\begin{proof}
If $k$ is a strong equivalence then $\A\to\B$ is a pseudo-pullback of $\tDelta$.  But since $\G\simeq\pathg{\G}$ is a strong equivalence over $\G\times\G$, it is also a pseudo-pullback of $\Delta$.
\end{proof}

\begin{lemma}\label{lem:disc_gpd_pspb}
For a groupoid $\G$ and object $X \in \hPP$, a pseudo-pullback as on the left below may be computed as a strict pullback as on the right,
\begin{equation}\label{diag:pseudo_vs_strict}
\begin{tikzcd}
	\X'  \ar[d] \ar[r]  &  \G  \ar[d,"\Delta"]
	\\
	\ul{X}  \ar[r,swap, "x"]  &  \G \times \G
\end{tikzcd}
\hspace{11ex}
\begin{tikzcd}
	X'  \ar[d] \ar[r]  \pbmark &  G_1  \ar[d]
	\\
	X  \ar[r, swap,"x"]  &  G_0 \times G_0
\end{tikzcd}
\end{equation}
where ${\X}'$ may be taken to be the discrete groupoid ${\X}' = \ul{X'}$.

When $\G \times_{\F} \G$ is a pseudo-pullback of a homomorphism $f : \G\to \F$ against itself, and $\Delta : \G \to \G \times_{\F}\G$ is the diagonal,  then the pseudo-pullback $\X'$ over a discrete groupoid $\ul{X}$,  
\[\begin{tikzcd}
	\X'  \ar[d] \ar[r]  &  \G  \ar[d,"\Delta"]
	\\
	\ul{X}  \ar[r,swap,"x"]  &  \G \times_{\F} \G
\end{tikzcd}
\]
is equivalent to a discrete groupoid $\X'\simeq \ul{X'}$, and may again be computed as an associated strict pullback.
Note that, in particular, $\pi_0\ul{X} \cong X$ for any discrete groupoid $\ul{X}$.
\end{lemma}

\begin{proof}
By Lemma~\ref{lemma:pspb_gpd}, the pseudo-pullback on the left in \eqref{diag:pseudo_vs_strict} can be computed by first replacing the diagonal $\Delta : \G \to \G
\times \G$ by the equivalent path-groupoid $\tDelta : \pathg{\G} \to \G\times\G$, and then taking a strict pullback in $\Gpd(\hPP)$.  We then apply the (right adjoint) forgetful functor $\Gpd(\hPP) \to \hPP$ to obtain the pullback diagram on the right.  But $\X'$ is discrete, since $\tDelta : \pathg{\G} \to \G\times\G$ has discrete fibers, so $\ul{X'} \cong \X'$.

The argument for the case $\G \times_{\F} \G$ is similar.
\end{proof}

\begin{definition}\label{def:ps_compact}
	A groupoid $\G = (G_1 \rightrightarrows G_0)$  in $\widehat{\PP}$ is called \emph{pseudo-compact} if there is an essentially-surjective-on-objects homomorphism $\ul{K}\to \G$ from the discrete groupoid $\ul{K}$ determined by a compact object $K \in\hPP$ (which we shall call a \emph{compact discrete groupoid}).
	Note that this internal condition can also be tested objectwise as a presheaf of groupoids.
\end{definition}

\begin{definition}[Coherent groupoid]\label{def:coh_gpd}
 A groupoid $\G = (G_1 \rightrightarrows G_0)$  in $\widehat{\PP}$ is called \emph{coherent} if the following conditions hold:
 \begin{enumerate}
 \item[(0)]  $\G$  is pseudo-compact,
 \item[(1)]   the diagonal $\Delta : \G \to \G \times \G$ is \emph{pseudo-compact},
 \item[(2)]   the second diagonal $\Delta_2 : \Delta \to \Delta \times \Delta$ is also \emph{pseudo-compact}.
 \end{enumerate}
\end{definition}

Condition (1) means that the pseudo-pullback of the diagonal $\Delta : \G \to \G \times \G$ over any homomorphism from a compact discrete groupoid $\ul{K}\to \G \times \G$ is again a compact discrete groupoid $\K' \in\hPP$\,: 
\[
\begin{tikzcd}
\K' \ar[d]  \ar[r] \pbmark & \G \ar[d, "\Delta"] \\
\ul{K} \ar[r]  & \G\times \G
\end{tikzcd}
\]
In virtue of Lemma~\ref{lem:disc_gpd_pspb}, ${\K}'$ can be assumed to be ${\K}' \cong \ul{K'}$  for $K'$ the strict pullback of $G_1 \to G_0 \times G_0$ along $k : K \to G_0 \times G_0$, and so the condition is simply that $\K' \simeq \ul{K'}$ is a compact object $K'$.

Similarly, Condition (2) involves the following pseudo-pullback of $\Delta$ against itself:
\[
\begin{tikzcd}
 \G \times_{\G \times \G} \G \ar[d]  \ar[r]  & \G  \ar[d, "\Delta"] \\
\G  \ar[r,swap, "\Delta"]  & \G \times \G
\end{tikzcd}
\]
The second diagonal (the ``paths between paths'') is then the canonical map $$\Delta_2 : \G \to  \G \times_{\G \times \G} \G\quad \text{(over $\G \times \G$)},$$
 and is again required to be pseudo-compact, in the previous sense.   

Finally, note that a \emph{discrete} coherent groupoid $\G$ is the discrete groupoid $\G = \underline{C}$ on a coherent object $C$, since a pseudo-pullback of discrete groupoids is a strict pullback.

\begin{lemma}\label{lemma:coh_gpd_equiv}
If $e: \F \to \G$ is an equivalence of groupoids in $\hPP$ and $\F$ is coherent, then so is $\G$.
\end{lemma}

\begin{proof}
If $\F$ is pseudo-compact, then $\G$ is, too, just by composing with~$e$.  

Suppose that $\Delta_\F : \F \to\F\times\F$ is pseudo-compact and consider the strict pullback of path groupoids
\[
\begin{tikzcd}
\pathg{\F} \ar[d, swap,"\tDelta"] \ar[r] \pbmark & \pathg{\G}  \ar[d, "\tDelta"] \\
 \F \times \F \ar[r, swap, "e \times e"]  & \G\times \G
\end{tikzcd}
\]

Given any compact $K$ and $k  : \ul{K} \to  \G\times \G$, we need to show that $k^*\pathg{\G}$ is (discrete and) compact. There is a cover $p:\yon{P} \onto K$, and by an argument similar to that for Proposition \ref{prop:CohIsEff}, it would now suffice show that the composite $kp : \ul{\yon{P}} \to \G\times \G$ lifts across $e \times e$, because $\pathg{\F}\to \F\times\F$ is assumed to be pseudo-compact.  Although $e \times e$ is not objectwise surjective, it is essentially surjective on objects and so there is a ``pseudo-lift'' $\ell : \ul{\yon{P}} \to \F\times\F$, meaning that there is a 2-cell $H : (e \times e)\ell \Rightarrow kp$. 
\[
\begin{tikzcd}
& \pathg{\F} \ar[d, swap,"\tDelta"] \ar[r] \pbmark & \pathg{\G}  \ar[d, "\tDelta"] \\
 & \F \times \F \ar[r, swap, "e \times e"] \ar[d, phantom, "\stackrel{H}{\Rightarrow}"]  & \G\times \G \\
 \ul{\yon{P}} \ar[r, swap, "p"] \ar[ru, "\ell"] & \ul{K} \ar[ru,swap, "k"] &
\end{tikzcd}
\]
It follows by Lemma \ref{lemma:postponed} below that there is an equivalence of groupoids  $h : \ell^*\pathg{\F} \simeq (kp)^*\pathg{\G}$ over $\ul{\yon{P}}$.  But $\ell^*\pathg{\F}$ is discrete compact, and $(kp)^*\pathg{\G}$ is discrete, and therefore $h$ is an iso, so $(kp)^*\pathg{\G}$ is also compact, as was required.  

The second diagonal of $\F$ is a pullback of that of $\G$ along the equivalence $e \times e$, so an analogous argument will show that it, too, is pseudo-compact.
\end{proof}
Recall that a functor $f: \F \to \G$ of groupoids is an \emph{isofibration} if the square
\[\begin{tikzcd}
\pathg{\F}  \ar[d,swap, "{\mathsf{cod}}"] \ar[r]  &  \pathg{\G}  \ar[d,"{\mathsf{cod}}"]
\\
\F  \ar[r,swap,"f"]  &  \G
\end{tikzcd}\]
is a pseudo-pullback (the same then holds with $\mathsf{dom}$ for $\mathsf{cod}$).

\begin{lemma}\label{lemma:postponed}
 Given an isofibration of groupoids $\X\to \G$, homomorphisms $f, g : \F \to \G$ and a 2-cell $\alpha : f \Rightarrow g$, there is a (strong) equivalence $f^*\X \to g^*\X$ over $\F$.
\end{lemma}

\begin{proof}
 Consider the universal case $u : \dom^*\X \to \cod^*\X$ over $\pathg{\G}$, which exists, and is an equivalence, because $\X\to \G$ is an iso-fibration.
 \[
\begin{tikzcd}
&\dom^*\X  \ar[rd] \ar[ld] \ar[rr, "u"] \ar[rr, swap, "\sim"] && \cod^*\X  \ar[ld] \ar[rd] & \\
\X \ar[rd]  && \ar[ld, "\dom"]   \pathg{\G}  \ar[rd,swap, "\cod"]  && \X \ar[ld] \\
& \G && \G &
\end{tikzcd}
\]
Now pull $u$ back along the unique lift $\tilde{\alpha} : \F \to \pathg{\G}$ of $\langle f, g\rangle : \F \to \G\times\G$ across $\tDelta : \pathg{\G} \to \G\times\G$ arising from $\alpha : f \Rightarrow g$.
\end{proof}

Coherent groupoids may be recognized as follows:

\begin{lemma}\label{lemma:coh_groupoid_recognize}
 Let  $\G = (G_1\rightrightarrows G_0)$ be a groupoid in $\hPP$. Then $\G$ is coherent in the sense of Definition \ref{def:coh_gpd}  if, and only if, there is a groupoid $\mathbb{K} = (K_1\rightrightarrows K_0)$ with an equivalence $\mathbb{K}\to \G$, such that:
 \begin{enumerate}
 \item $K_0$ is a compact object, i.e., one with a cover $\yon{P} \onto K_0$,
 \item the canonical map $K_1 \to K_0\times K_0$ is compact, 
 \item taking the following pullback
\[
\begin{tikzcd}
 K_1\times_{K_0\times K_0} K_1\ar[d]  \ar[r] \pbmark & K_1  \ar[d] \\
K_1  \ar[r,swap]  & K_0\times K_0
\end{tikzcd}
\]
the diagonal $K_1\to K_1\times_{K_0\times K_0} K_1$ (over $K_0\times K_0$) is compact.
\end{enumerate}
\end{lemma}

\begin{proof}
Suppose $\G = (G_1\rightrightarrows G_0)$ is coherent.  Since $\G$ is pseudo-compact, there is a compact object $K_0$ and an essentially surjective homomorphism $\ul{K_0} \to \G$.
Consider the pseudo-pullback of the diagonal of $\G$ to $\ul{K_0}\times \ul{K_0}$,
\[
\begin{tikzcd}
\K_1  \ar[d] \ar[r] & \G  \ar[d] \\
 \ul{K_0} \times \ul{K_0} \ar[r]  & \G\times \G
\end{tikzcd}
\]
where we may assume that $\K_1 = \ul{K_1}$ is discrete thanks to Lemma~\ref{lem:disc_gpd_pspb}.
Since $\Delta : \G\to \G\times\G$ is pseudo-compact and $K_0\times K_0$ is compact, $K_1$ is a compact object.  Moreover, $K_1 \rightrightarrows K_0$ is a groupoid structure on the object $K_0$, as it can be seen by constructing $K_1$ as the strict pullback on the right in Lemma~\ref{lem:disc_gpd_pspb}.

Now let $\mathbb{K} = (K_1 \rightrightarrows K_0)$, so that we indeed have $\mathbb{K} \simeq \G$, since $\ul{K_0} \to \G$ was essentially surjective on objects and $\mathbb{K} \to \G$ is fully faithful by construction.
Observe that the canonical map $K_1 \to K_0\times K_0$ is compact by Lemma~\ref{lem:disc_gpd_pspb}, since by construction it is the pseudo-pullback of the pseudo-compact map $\Delta : \G\to \G\times\G$.
It remains only to show that the map $K_1\to K_1\times_{K_0\times K_0} K_1$ (over $K_0\times K_0$) is compact.
\[
\begin{tikzcd}
K_1 \ar[r] \ar[rd, equals] & K_1\times_{K_0\times K_0} K_1\ar[d]  \ar[r] \pbmark & K_1  \ar[d] \\
& K_1  \ar[r,swap]  & K_0\times K_0
\end{tikzcd}
\]
But the map $K_1\to K_1\times_{K_0\times K_0} K_1$ is a pseudo-pullback of the second diagonal $\Delta_2 : \G \to  \G \times_{\G \times \G} \G$ (over $\G \times \G$) along $K_0 \times K_0 \to \G \times \G$, and is therefore also compact.

For the converse, suppose given an equivalence of groupoids $\mathbb{K} \to \G$ and suppose that for $\mathbb{K} = (K_1\rightrightarrows K_0)$ the object $K_0$ is compact and the canonical maps $K_1 \to K_0\times K_0$ and $K_1\to K_1\times_{K_0\times K_0} K_1$ are compact.  By Lemma \ref{lemma:coh_gpd_equiv} it suffices to show that $\mathbb{K}$ is coherent. 
This follows easily from Lemma \ref{lem:disc_gpd_pspb}.
\end{proof}

\section{The main result}

Ultimately, we shall require a Quillen model structure on the category $\mathsf{Gpd}(\hPP)$ in order to have a model of (1-trucated) HoTT (and so a $(2,1)$-topos); but for the present purpose, we need it in order to have a well-behaved subcategory of $0$-types for the comparison \eqref{diag:subcats}. 

There are various different model structures that can be put on the category $\mathsf{Gpd}(\EE)$ of internal groupoids in a Grothendieck topos $\EE$, see~\cite{JTStacks,EKVdL2005,Hollander2008,HollanderThesis}.  In the case of presheaves $\EE = \widehat{\PP}$, where
\[
 \mathsf{Gpd}(\,\hPP\,) \cong [\op{\PP},  \mathsf{Gpd}]\,,
\]
we shall use a model structure for which the weak equivalences are the pointwise equivalences of groupoids, and the fibrant objects are the \emph{stacks}, a homotopy theoretic generalization of sheaves.  The resulting model category will be Quillen equivalent to {e.g.} the Joyal-Tierney ``strong stacks'' model structure, which can also be described in terms of categories fibered in groupoids (or ``lax presheaves'' of groupoids), satisfying a descent condition (see~\cite{HollanderThesis}).  
Not only will the subcategory of coherent $0$-types then be equivalent to the coherent presheaves, and therefore to the realizability 1-topos $\Eff$, but the subcategory $\mathsf{Coh}\mathsf{Gpd}(\,\hPP\,) \hook \mathsf{Gpd}(\,\hPP\,)$ of all coherent groupoids should then admit a model of HoTT, and will therefore be a reasonable candidate for the realizability $(2,1)$-topos $\Eff_2$ (about which we shall say a bit more below).  

\begin{theorem}[\cite{JTStacks}]\label{thm:modelstructure}
 For any small category~$\catC$, the following classes of maps determine a model structure on the category $[\op{\catC},\mathsf{Gpd}]$ of presheaves of groupoids on $\catC$.
 \begin{itemize}
 \item The \emph{cofibrations} are the functors that are objectwise injective on objects.
 
  \item The \emph{weak equivalences} are the functors that are objectwise weak equivalences of groupoids: functors that are fully faithful and essentially surjective on objects.
  
  \item The \emph{fibrations} are the maps with the right-lifting property with respect to the trivial cofibrations (those that are weak equivalences).    
  \end{itemize}
\end{theorem}

As a localization of simplicial presheaves, this agrees with the ``injective model structure'' used in \cite{Shulman2019} for the construction of univalent universes.  As explained there, the fibrations may be described as the objectwise fibrations of groupoids that are also algebras for the \emph{cobar monad}. In the present case of presheaves of groupoids the fibrant replacement amounts to taking a homotopy limit of descent data for \emph{all} composable arrows, as explained in \cite{Shulman2019}.   This model structure also agrees with that for \emph{strong stacks} first established by Joyal and Tierney in \cite{JTStacks}, since the cofibrations and weak equivalences agree with the ones specified there.  

\begin{lemma}[{\cite[Lemma~1]{JTStacks}}]\label{lemma:triv_fib_eqv}
A trivial fibration is a strong equivalence.
\end{lemma}

\begin{proof}
We obtain a section using the fact that every object is cofibrant.
It is a weak inverse since a trivial fibration is necessarily fully faithful, because the functor $\mathbbm{1}+\mathbbm{1} \to \mathbbm{2}$ is a cofibration.
\end{proof}

\begin{corollary}\label{corol:coh_gpd_equiv_back}
Let $e: \F \to \G$ be a trivial fibration (i.e.\ both an equivalence of groupoids and a fibration in the model structure).
Then $\F$ is coherent if and only if is $\G$ is coherent.
\end{corollary}

\begin{proof}
Lemma~\ref{lemma:triv_fib_eqv} allows Lemma~\ref{lemma:coh_gpd_equiv} to apply in both directions.
\end{proof}

\begin{proposition}\label{prop:modelstructure_coh}
The two factorization systems of the model structure on $[\op{\PP},\mathsf{Gpd}]$ from Theorem~\ref{thm:modelstructure} restrict to the full subcategory $\mathsf{Coh}\mathsf{Gpd}(\,\hPP\,)$ of coherent presheaves of groupoids.
\end{proposition}

\begin{proof}
It suffices to verify that the two factorizations produce coherent groupoids when applied to functors between coherent groupoids.
Consider a factorisation $f=p\circ i: \F \cof \HH \fib \G$ of a functor $f$ into a cofibration $i$ followed by a fibration $p$.
If $i$ is a weak equivalence, then $\HH$ is coherent by Lemma~\ref{lemma:coh_gpd_equiv}.
If $p$ is a weak equivalence, then $\HH$ is coherent by Corollary~\ref{corol:coh_gpd_equiv_back}.
\end{proof}

Recall that a functor of groupoids $f: \F \to \G$ is a \emph{discrete fibration}  if the square
\[\begin{tikzcd}
F_1  \ar[d,"{\partial_1}"'] \ar[r,"f_1"]  &  G_1  \ar[d,"{\partial_1}"]
\\
F_0  \ar[r,"f_0"]  &  G_0
\end{tikzcd}\]
is a pullback.
Note that, for every groupoid $\G$, the functor $\pathg{\G}\to \G \times \G$ is a discrete fibration.

\begin{proposition}[{\cite[Proposition~1]{JTStacks}}]\label{prop:disc_fib}
A discrete fibration has the unique right lifting property against trivial cofibrations, and is therefore a fibration in the sense of the model structure.
\end{proposition}

\begin{corollary}\label{corol:path_fib}
For every groupoid $\G$, the functor $\pathg{\G}\to \G \times \G$ is a fibration.
\end{corollary}

\begin{proposition}\label{prop:fib_eqv}
If $\F$ is fibrant, every equivalence $e: \F \to \G$ is a strong equivalence.
\end{proposition}

\begin{proof}
By Lemma~\ref{lemma:triv_fib_eqv} it is enough to prove the statement when $e$ is also a cofibration.
We thus obtain a retraction $r: \G \to \F$ of $e$ since $\F$ is fibrant.
It is a weak inverse because we can fill the square below by Corollary~\ref{corol:path_fib}.
\[\begin{tikzcd}[column sep=2ex]
\F \ar[d,"e"',tail,"\sim"{anchor=center,rotate=-90,yshift=.5ex}] \ar[r,"e"]
&[2ex]  \G  \ar[r]  &  \pathg{\G}  \ar[d,two heads]
\\
\G  \ar[rr,"\pair{e\circ r,\id{}}"'] \ar[urr,dashed]  &&  \G \times \G
\end{tikzcd}\]
\end{proof}

We make use of the model structure on $[\op{\PP},\mathsf{Gpd}]$ to determine the notion of a \emph{$0$-type} as a fibrant object $X$ for which the (first) diagonal $\Delta : X\to X \times X$ is a ``homotopy monomorphism'', meaning that the diagram
\[
\begin{tikzcd}
 X \ar[d, equals] \ar[r, equals] & X  \ar[d, "\Delta"] \\
 X \ar[r,swap, "\Delta"]  & X \times X
\end{tikzcd}
\]
is a homotopy pullback.

\begin{lemma}\label{lemma:fibzeroeqrel}
	Let $\G=(G_1 \toto G_0)$ be a fibrant 0-type.
	Then $\G$ is an equivalence relation, meaning that the map $G_1 \to G_0 \times G_0$ is monic.
\end{lemma}

\begin{proof}
	Consider the commutative cube below, where the front face is a strict pullback of groupoids. In particular, objects of $\mathbb{H}$ consist of pairs of parallel isomorphisms $g_1,g_2\colon x_1 \toto x_2$ in $\G$.
	\[\begin{tikzcd}
		\G  \ar[dd,swap,"\id{}"] \ar[rr,"\id{}"] \ar[dr]
		&&  \G \ar[dd,"\Delta"{near end}] \ar[dr,"\rotatebox{-40}{$\sim$}"{yshift=-1ex}]
		&\\
		&  \mathbb{H}  \ar[rr, {crossing over}]  &&  \pathg{\G}  \ar[dd,two heads]
		\\
		\G  \ar[dr,"\rotatebox{-40}{$\sim$}"{yshift=-1ex}] 
		\ar[rr,swap,"\Delta"{near end}] && \G \times \G  \ar[dr,"\id{}"] 
		&\\
		&  \pathg{\G}  \ar[rr,two heads]   \ar[from=uu, {crossing over}]&&  \G \times \G
	\end{tikzcd}\]
	
	Since $\G$ is a 0-type, the back square is a homotopy pullback and, since $\G\times \G$ is fibrant and $\pathg{\G}\twoheadrightarrow\G\times \G$ is a fibrant replacement of $\Delta$ by Corollary~\ref{corol:path_fib}, the canonical comparison $\G\to \mathbb{H}$ is a weak equivalence.  
	By Proposition~\ref{prop:fib_eqv}, the comparison $\G\to \mathbb{H}$ is then a strong equivalence. In particular, every object $(g_1,g_2)$ of $\mathbb{H}$ is isomorphic to one of the form $(\mathrm{id}_x,\mathrm{id}_x)$ for some $x\in\G$.
	It follows that $g_1=g_2$ as required.
\end{proof}

We can now prove our main result, namely:

\begin{theorem}
 Let $\G = (G_1\rightrightarrows G_0)$ be a coherent groupoid and a fibrant $0$-type. Then $\G$ is equivalent to (the discrete groupoid arising from) a coherent object.
\end{theorem}

\begin{proof}
By Lemma \ref{lemma:coh_groupoid_recognize} we may assume, without loss of generality,  that $G_0$ is compact and therefore has a cover $\yon{P} \onto G_0$ and that $G_1 \to G_0 \times G_0$ is a compact map.  Thus its pullback $K$ to $\yon{P}\times \yon{P}$ has a cover $\yon{Q}\onto K$. 
\[
\begin{tikzcd}
\yon{Q} \ar[r, two heads] \ar[rd] & K \ar[d]  \ar[r] \pbmark & G_1  \ar[d] \\
& \yon{P} \times\yon{P} \ar[r,two heads]  & G_0\times G_0
\end{tikzcd}
\]
But since $\G$ is a fibrant $0$-type, by Lemma \ref{lemma:fibzeroeqrel} we may assume that the map $G_1 \to G_0 \times G_0$ is monic.  Therefore so is its pullback $K\to \yon{P} \times\yon{P}$.  Thus $K$ is an assembly. 

Now let $$G = \pi_0(\G) = G_0/G_1\,.$$   We then have a pullback square on the right below, and so $K\toto \yon{P}$ is the kernel pair of the composite epi $\yon{P} \onto G$. 
\begin{equation}\label{diag:theorem2}
\begin{tikzcd}
\yon{Q} \ar[r, two heads] \ar[rd] & K \ar[d, tail]  \ar[r] \pbmark & G_1  \ar[d] \ar[r]  \pbmark & G \ar[d, tail]  \\
& \yon{P} \times\yon{P} \ar[r,two heads]  & G_0\times G_0  \ar[r,two heads]  &  G \times G
\end{tikzcd}
\end{equation}
Since $K$ and $\yon{P}$ are both assemblies, and we have an exact presentation $K \rightrightarrows P \onto G$, the quotient  $G$ is a coherent object.  

Finally, for the discrete groupoid $\ul{G}$, we have an equivalence $\G \simeq \ul{G}$ and therefore a weak equivalence in the model structure, since the quotient map $G_0 \onto G$ makes the homomorphism $\G \to \ul{G}$ surjective on objects, and it is fully faithful by the pullback square on the right in \eqref{diag:theorem2}.
\end{proof}

Since every discrete groupoid is a $0$-type, and the discrete one arising from a coherent object is plainly coherent as a groupoid, we have the desired result:
\begin{corollary}
 The full sub(2-)category of fibrant $0$-types in $\mathsf{Coh}\mathsf{Gpd}(\hPP)$ is equivalent, as a 1-category, to the coherent objects in $\hPP$, and therefore to the effective topos,
 \[
 {\mathsf{Coh}\mathsf{Gpd}}(\hPP)_0\, \simeq\, {\mathsf{Coh}}(\hPP)\, \simeq\, \Eff\,.
 \]
\end{corollary}
\medskip

\section{Related and Future work}

Recent work of Anthony Agwu \cite{AA} builds a realizability model of HoTT using groupoids in a category of assemblies, while the work of Calum Hughes \cite{Hughes2025} can be used to do the same for groupoids in a realizability topos.  As explained in Problem~\eqref{prob:zeroty} of our introduction, these models are expected to add further $0$-types to the realizability 1-topos.  It may be that they can be localized to give a model equivalent to ours, but we are not currently aware of any such results.

As already explained, the model structure on $\mathsf{Gpd}(\hPP)$ specified in Theorem \ref{thm:modelstructure} is a restriction to the 1-types of the injective model structure on simplicial presheaves, which was used in~\cite{Shulman2019} to build models of HoTT with univalent universes. We expect the fibrant objects in our category of ``coherent 1-stacks'' to also be closed under exponentials (and $\Pi$-types of fibrations), and that there is a univalent universe of $0$-types.  In the present case of realizability, we even expect there to be an \emph{impredicative} such universe.  We leave this for future work (in progress as~\cite{AAB}).

\subsection*{Acknowledgements}

We are grateful to Mathieu Anel, Reid Barton, Jonas Frey, Pino Rosolini, Michael Shulman, and Andrew Swan, for sharing their ideas and advice, for essential contributions to the definitions and proofs, and for detailed comments on a draft.

We are also grateful to the University of Genoa for supporting the first author's visit, during which this work was begun.
This material is based upon work supported by the Air Force Office of Scientific Research under awards number FA9550-21-1-0009 and  FA9550-23-1-0434 (for the first author), and by the Italian Ministry of University and Research (MUR) under award ``PNRR - Young Researchers - SOE 0000071'' (for the second author).

%
%

\input{eff2.bbl}

\end{document}

%% file: eff2.bbl
\newcommand{\etalchar}[1]{$^{#1}$}